\documentclass[12pt, reqno]{amsart}

\usepackage{amssymb, amsmath, amsthm}
\usepackage[backref]{hyperref}
\usepackage[alphabetic,backrefs,lite]{amsrefs}
\usepackage{amscd}   
\usepackage{fullpage}
\usepackage[all]{xy} 

\DeclareFontEncoding{OT2}{}{} 

\usepackage[usenames,dvipsnames]{color}

\newtheorem{lemma}{Lemma}[section]
\newtheorem{theorem}[lemma]{Theorem}

\newtheorem{prop}[lemma]{Proposition}
\newtheorem{cor}[lemma]{Corollary}

\newtheorem{claim*}{Claim}

\newtheorem{remark}[lemma]{Remark}
\newtheorem{remarks}[lemma]{Remarks}
\newtheorem{thm}[lemma]{Theorem}
\newtheorem{defn}[lemma]{Definition}

\newcommand{\Q}{{\mathbb Q}}

\newcommand{\Z}{{\mathbb Z}}

\newcommand{\BBl}{{\mathbb B}_{\ell, \infty}}

\newcommand{\Dtilde}{{\widetilde{D}}}
\newcommand{\Ktilde}{{\widetilde{K}}}
\newcommand{\Ftilde}{{\widetilde{F}}}

\newcommand{\calO}{{\mathcal O}}

\newcommand{\OO}{{\mathcal O}}

\newcommand{\fraka}{{\mathfrak a}}
\newcommand{\frakb}{{\mathfrak b}}

\newcommand{\frakl}{{\mathfrak l}}

\newcommand{\frakn}{{\mathfrak n}}

\newcommand{\frakp}{{\mathfrak p}}

\newcommand{\frakD}{{\mathfrak D}}

\newcommand{\frakN}{{\mathfrak N}}

\newcommand{\frakP}{{\mathfrak P}}

\newcommand{\pp}{{\mathfrak p}}

\DeclareMathOperator{\Tr}{Tr}

\DeclareMathOperator{\End}{End}

\DeclareMathOperator{\Norm}{Norm}

\DeclareMathOperator{\Disc}{Disc}

\DeclareMathOperator{\CM}{CM}

\numberwithin{equation}{section}
\numberwithin{table}{section}

\title{Comparing arithmetic intersection formulas for denominators of Igusa class polynomials}
\author{Jacqueline Anderson}
\address{Department of Mathematics, Box 1917, Brown University, Providence, RI
			02912, USA}
\email{jackie@math.brown.edu}
\urladdr{http://math.brown.edu/\~{}jackie}

\author{Jennifer S. Balakrishnan}
\address{Department of Mathematics, Harvard University, 1 Oxford Street, Cambridge, MA 02138, USA}
\email{jen@math.harvard.edu}
\urladdr{http://www.math.harvard.edu/\~{}jen/}

\author{Kristin Lauter}
\address{Microsoft Research, 1 Microsoft Way, Redmond, WA 98062, USA}
\email{klauter@microsoft.com}
\urladdr{http://research.microsoft.com/en-us/people/klauter/default.aspx}

\author{Jennifer Park}
\address{Department of Mathematics, MIT, 77 Massachusetts Avenue, Cambridge, MA 02139, USA}
\email{jmypark@math.mit.edu}
\urladdr{http://math.mit.edu/\~{}jmypark/}

\author{Bianca Viray}
\address{Department of Mathematics, Box 1917, Brown University, Providence, RI
			02912, USA}
\email{bviray@math.brown.edu}
\urladdr{http://math.brown.edu/\~{}bviray}

\thanks{The second author was supported by NSF grant DMS-1103831.  The fourth author was partially supported by NSF grant DMS-1069236 and by a NSERC PGSD grant. The last author was partially supported by NSF grant DMS-1002933 and ICERM}

\date{}

\begin{document}

	\begin{abstract}
		Bruinier and Yang conjectured a formula for intersection numbers on an arithmetic Hilbert modular surface, and as a consequence obtained a conjectural formula for $\CM(K).\textup{G}_1$ under strong assumptions on the ramification in $K$.  Yang later proved this conjecture under slightly stronger assumptions on the ramification.  In recent work,  Lauter and Viray proved a different formula  for $\CM(K).\textup{G}_1$ for primitive quartic CM fields with a mild assumption, using a method of proof independent from that of Yang.   
In this paper we show that these two formulas agree, for a class of primitive quartic CM fields which is slightly larger than the intersection of the fields considered by Yang and Lauter and Viray.  Furthermore, the proof that these formulas agree does \emph{not} rely on the results of Yang or Lauter and Viray. As a consequence of our proof, we conclude that the Bruinier-Yang formula holds for a slightly largely class of quartic CM fields $K$ than what was proved by Yang, since it agrees with the Lauter-Viray formula, which is proved in those cases. The factorization of these  intersection numbers has applications to cryptography: precise formulas for them allow one to compute the denominators of Igusa class polynomials, which has important applications to the construction of genus $2$ curves for use in cryptography.
		
	\end{abstract}

	\maketitle

	\section{Introduction}

		In this paper we study the relationship between two formulas proved for arithmetic intersection numbers on the Siegel moduli space of principally polarized abelian surfaces. Specifically, these are formulas for the arithmetic intersection of the CM points of $K$, denoted by $\CM(K)$, with the Humbert surface $\textup{G}_1$, which parametrizes abelian surfaces isomorphic to a product of elliptic curves with the product polarization; the $\ell$-part of this arithmetic intersection number is denoted $(\CM(K).\textup{G}_1)_{\ell}$.  

		The study of these particular intersection numbers was largely motivated by applications to cryptography.  In order to generate genus $2$ curves over a finite field whose Jacobians have prime order, the {\it CM method} proceeds by computing the minimal polynomials of the invariants of the genus $2$ curves with CM by a primitive quartic CM field $K$.  These minimal polynomials are analogous to the Hilbert class polynomials for imaginary quadratic fields $K$. Indeed, Igusa defined a collection of invariants for genus $2$ curves and proved expressions for them in terms of quotients of Siegel modular forms. For genus $2$ curves with complex multiplication (CM) by a primitive quartic CM field $K$, these invariants lie in the Hilbert class field of the reflex field of $K$, and their minimal polynomials, {\it Igusa class polynomials}, have coefficients which are rational, not necessarily integral as is the case for Hilbert class polynomials related to invariants of elliptic curves.

		Ignoring cancellation with numerators, the primes which appear in the denominators of Igusa class polynomials are those which appear in $\left(\CM(K). G_1\right)$, the arithmetic intersection on the Siegel moduli space of the divisor of the Siegel modular form $\chi_{10}$ with the CM points of $K$. In~\cite{GL}, it was proved that these primes are those $\ell$ for which there is a solution to an {\it Embedding Problem}, that is, there exists an embedding of $\OO_K$ into $M_2(\BBl)$ with certain properties.  Studying this embedding problem,~\cite{GL} gave a bound on the primes which can appear, and~\cite{GL10} gave a bound on the powers to which they can appear.

		At the same time, Bruinier and Yang, using methods from Arakelov intersection theory, gave a conjectural exact formula for the factorization of the denominators under certain conditions on the ramification in the primitive quartic CM field $K$~\cite{BY}. They assume that the discriminant of $K$ is $D^2\Dtilde$, where $D$ and $\Dtilde$ are both primes congruent to $1 \pmod{4}$. In~\cite{Yang-delta1, Yang-HilbertModularAndFaltings}, Yang proved the conjectured intersection formula assuming the ring of integers of $K$ is generated by one element over the ring of integers of the real quadratic subfield.  Yang's proof uses results of Gross-Keating, and then computes local densities by evaluating certain local integrals over the quaternions.

		In practice, very few primitive quartic CM fields have ramification of such restricted form. In~\cite{WIN}, Grundman, Johnson-Leung, Salerno, Wittenborn, and the third and last author studied all $13$ quartic cyclic CM fields in van Wamelen's tables of CM genus 2 curves defined over $\Q$, compared denominators with the number of solutions to the Embedding Problem and Bruinier and Yang's formula, and found that the Bruinier-Yang formula does not hold in general as stated when the assumptions on the ramification of $K$ are relaxed.  For applications to the computation of genus $2$ curves for cryptography, it is important to have a precise formula for the denominators of Igusa class polynomials which holds for general primitive quartic CM fields.

		In~\cite{LV-Igusa}, the third and last author proved a formula for $\left(\CM(K). G_1\right)$ for primitive quartic CM fields $K$ with almost no assumptions on $K$.  A simplified version of this formula holds with an extra mild assumption.  The proof of their formula follows from parameterizing solutions to the Embedding Problem by pairs of endomorphisms of a supersingular elliptic curve $E$, $x, u \in \End(E)$ with a fixed norm and trace.  This, in turn, is related to a counting problem studied by Gross and Zagier in their formula for the factorization of differences of singular moduli.

		The formula given by Bruinier and Yang is strikingly similar to the simplified version of the formula in~\cite{LV-Igusa}.  Indeed, both formulas involve two nested sums where the summand is a product that includes the number of ideals of a given norm.  However, the Bruinier-Yang formula counts ideals in a quartic CM field, whereas the Lauter-Viray formula counts ideals in an imaginary quadratic order.  

		In this paper, we show that the formulas of Bruinier-Yang (BY) and Lauter-Viray (LV) agree, \emph{without} using that the formulas compute the same arithmetic intersection number.  As a consequence of our result, we conclude that the BY formula holds for a slightly larger class of quartic CM fields $K$ than what was proved by Yang.  See~\S\ref{sec:EqualityIndices} for more details.

		\subsection{Idea of proof}
			The BY formula sums over elements in $\Ftilde$, the real quadratic subfield of $\Ktilde$, the reflex field of $K$, counting the number of ideals of $\Ktilde$ with certain norms, with a certain multiplicity.  The LV formula sums over certain integers which turn out to be in one-to-one correspondence with the elements of $\Ftilde$ which arise in BY, under the assumptions on the ramification of $K$ (see Proposition~\ref{prop:EquivCondOnn}).  For each such integer $n$ in the LV formula, two related imaginary quadratic fields are defined, with suborders of discriminants $d_u = d_u(n)$ and $d_x = d_x(n)$ respectively.  The LV formula counts ideals in $\Z[(d_u + \sqrt{d_u})/2]$ with norm equal to $N$, a quantity related to $n$, with certain multiplicities. The heart of our proof of the equality of these two formulas is Proposition~\ref{prop:Splitting}, which shows how the splitting behavior of certain primes in the quadratic extension $\Ktilde/\Ftilde$ is related to the splitting behavior of certain primes in the quadratic extensions $\Q(\sqrt{d_u})$ and $\Q(\sqrt{d_x})$. Thus the results of this paper can be viewed as a kind of ``reciprocity'' between splitting behavior of certain primes in different quadratic extensions.  The resulting equality of local factors in the BY and LV formulas also involves the multiplicities which appear in the LV formula arising from genus theory, and indeed our proofs rely heavily on computations of related Hilbert symbols.

		\subsection{Outline of paper}
			In \S\S\ref{sec:BY} and \ref{sec:LV}, we precisely state the BY formula and the LV formula, respectively.  We also prove that these formulas can be expressed as a product of local factors, which will be instrumental in the proof of our main result.

			Both formulas rely on a relative integral basis for the ring of integers of $K$ over the ring of integers of the real quadratic subfield $F$.  In~\S\ref{sec:IntBasis} we give the possibilities for this integral basis under the assumption that $D$ is prime and $\Dtilde$ is squarefree.  

			We precisely state our main result in~\S\ref{sec:EqualityIndices} and begin the proof by showing that the BY formula and LV formula both sum over the same indices.  The crux of the proof is in~\S\ref{sec:EqualitySummands}, where we show that the summands in the BY formula and LV formula agree by comparing the local factors.  
		\subsection{Notation}
			Let $F/\Q$ denote a real quadratic field, and let 
			\[
				D = \Disc_{F/\Q}(\calO_F).
			\]
			Let $A, B\in \frac12\Z$ be such that $A + B\sqrt{D}\in\OO_F$ and $K = F(\sqrt{A + B\sqrt{D}})$ is a totally imaginary quadratic extension of $F$. Throughout this paper, $K$ will be assumed to be a \emph{primitive} quartic CM field, which is the case if it is either non-Galois or Galois cyclic \cite[Ch. II, \S 8.4]{Shimura}. Write $$\Dtilde = \Norm_{F/\Q}(\Disc_{K/F}(\OO_K))$$ and write $\Ftilde = \Q(\sqrt{\Dtilde})$. Let $\Ktilde$ denote the reflex field of $K$. After possibly changing the CM-type of $K$, one can check that $\Ktilde = \Ftilde(\sqrt{2A + 2\sqrt{A^2 - B^2D}}).$  Denote the relative discriminant of $\Ktilde/\Ftilde$ by 
			\[
				\frakD_{\Ktilde/\Ftilde}=\Disc_{\Ktilde/\Ftilde}(\calO_\Ktilde).
			\]

	\section*{Acknowledgements}
		This project was started during the Women in Numbers 2 workshop at the Banff International Research Station; we thank the workshop organizers, Chantal David, Matilde Lal\'in, and Michelle Manes, and the staff at BIRS for their support. 

	\section{The Bruinier-Yang formula}\label{sec:BY}

		In this section, we describe the formula for $(\CM(K).\textup{G}_1)_{\ell}$ that was conjectured by Bruinier and Yang~\cite{BY} and later proved by Yang~\cite{Yang-HilbertModularAndFaltings}. One of the factors (which we denote by $R_{\Ktilde/\Ftilde}$) in this intersection formula is multiplicative, so it makes sense to express it in terms of local factors over each prime. The key result of this section is Theorem \ref{T: localBY}, which does exactly this. We start by recalling the necessary definitions to state the formula of Bruinier and Yang.

		\subsection{Yang's Theorem}
			\begin{defn}
				Let $\fraka \subseteq \Ftilde$ be an ideal.  We define 
				\[
					R_{\Ktilde/\Ftilde}(\fraka) := 
					\#\{\frakn \subset \calO_{\Ktilde} \mid 
					\Norm_{\Ktilde/\Ftilde}(\frakn) = \fraka \}.
				\]
			\end{defn}

			\begin{theorem}~\cite[Thm. 1.2]{Yang-HilbertModularAndFaltings}
				\label{thm:BYFormula}
				Assume that $D$ and $\Dtilde$ are congruent to $1$ modulo $4$ and prime.  Further assume that $\calO_K = \calO_F[(w + \sqrt{A + B \sqrt D})/2]$ for some $w\in\OO_F$; this implies that $\Dtilde = A^2 - B^2D$.  Then for each rational prime $\ell$, ${(G_1.CM(K))_{\ell}}/({\log \ell})$ equals
				\begin{equation}\label{eq:BYFormula}
					\sum_{\delta = \frac{D - x^2}4 \in \Z_{\geq0}} \qquad 
					\sum_{\substack{n \textup{ s.t. } \frac{n + \delta\sqrt{\Dtilde}}{2D} \in \frakD_{\Ktilde/\Ftilde}^{-1}\\ |n|<\delta\sqrt{\Dtilde}}} B_{\frac{n + \delta\sqrt{\Dtilde}}{2D}}(\ell),
				\end{equation}
				where $x$ is some integer, and
				\[
				 B_t(\ell) = \sum_{\frakl|\ell}
					\begin{cases}
					0 & \textup{ if } \mathfrak{l} \textup{ splits in } \widetilde{K}\\
					\frac12(v_{\mathfrak{l}}(t)+1)R_{\Ktilde/\Ftilde}(t \frakD_{\widetilde{K}/\widetilde{F}}\mathfrak{l}^{-1})f(\mathfrak{l}/\ell) & \textup{ otherwise},
					\end{cases}
				\]
				where the sum ranges over prime ideals $\frakl$ of $\OO_\Ftilde$ lying over $\ell$ and $f(\frakl/\ell)$ denotes the inertial degree of $\frakl$ over $\calO_F$.
			\end{theorem}

			\begin{remarks}\ 
				\label{R: BruinierYangNorm}
				\begin{enumerate}
					\item According to~\cite{KW}, if $A^2 -B^2 D$ is not a square, then $K$ is primitive, so this assumption is certainly satisfied if $\Dtilde$ is prime or squarefree.
					\item In Lemma 4.4 and Corollary 4.5 of \cite{BY}, it is proved that under the assumptions of the above theorem, $\Norm(\frakD_{\Ktilde/\Ftilde}) = D$. 
					\item If $D = 5$ so that the only value for $\delta$ is $1$, then Yang~\cite{Yang-delta1} showed that the same statement holds if the assumption that $\Dtilde$ is $1$ modulo $4$ and prime is replaced with the assumption that $\Dtilde$ is $1$ modulo $4$ and squarefree.
				\end{enumerate}
			\end{remarks}

		\subsection{A local interpretation of the Bruinier-Yang formula}
			\begin{defn}
			Let $p$ be a rational prime, and let $\fraka$ be an ideal in $\Ftilde$. Then we define: 
				\[
					\varepsilon_{\Ktilde/\Ftilde}(p, \fraka) := 
						\prod_{\pp|p, v_{\pp}(\fraka)> 0}
						\begin{cases}
							\frac12(1 + (-1)^{v_{\pp}(\fraka)}) & 
								\textup{if }\pp\textup{ is inert in }\Ktilde\\
							v_{\pp}(\fraka) + 1 & 
								\textup{if }\pp\textup{ is split in }\Ktilde\\
							1 & \textup{otherwise}.
						\end{cases}
				\]
			\end{defn}
			Let $n\in\Z$ be such that $|n|< \delta\sqrt{\Dtilde}$ and such that $\frakN := (\frac{n + \delta\sqrt{\Dtilde}}{2D})$ divides $\frakD_{\Ktilde/\Ftilde}^{-1}$.  Assume that $\Norm(\frakD_{\Ktilde/\Ftilde}) = D$.  Then we define 
			\[
				N := \Norm(\frakN\frakD_{\Ktilde/\Ftilde}) = \frac{\delta^2\Dtilde - n^2}{4D}.
			\]

			\begin{thm}\label{T: localBY}\label{thm:BYLocal}
				  Let $\ell$ be a rational prime. If $v_{\frakl}(\frakN\frakD_{\Ktilde/\Ftilde}) = 0$ for all primes $\frakl|\ell$ in $\OO_{\Ftilde}$, then $B_{\frakN}(\ell) = 0$.
				Assume that there exists exactly one prime $\frakl|\ell$ in $\OO_{\Ftilde}$ such that $v_{\frakl}(\frakN\frakD_{\Ktilde/\Ftilde}) >0$ and this prime $\frakl$ is unramified in $\Ktilde$. Then
				\begin{equation}\label{eq:BYFormulaLocal}
					B_{\frakN}(\ell) = 
					\begin{cases}
						\frac{v_{\mathfrak{l}}(\frakN)+1}2f(\mathfrak{l}/l)\prod_{p|N, p\neq \ell}\varepsilon_{\Ktilde/\Ftilde}(p, \frakN  \frakD_{\Ktilde/\Ftilde}) & \textup{if }\frakl\textup{ is inert in }\Ktilde/\Ftilde,\\
						&\textup{and }v_{\frakl}(\frakN)\equiv 1 \bmod 2\\
						0 &\textup{otherwise}.
					\end{cases}
				\end{equation}
			\end{thm}
			\begin{proof}
				If $v_{\frakl}(\frakN \frakD_{\Ktilde/\Ftilde}) = 0$ for all primes $\frakl$ in $\OO_{\Ftilde}$ lying over $\ell$, then $\frakl^{-1}\frakN\frakD_{\Ktilde/\Ftilde}$ is not integral so it cannot be the norm of an integral ideal in $\OO_{\Ktilde}$.  Thus, $R_{\Ktilde/\Ftilde}(\frakl^{-1}\frakN\frakD_{\Ktilde/\Ftilde}) = 0$ for all $\frakl |\ell$ and so $B_{\frakN}(\ell) = 0$.

				Henceforth, assume that there exists exactly one prime $\frakl|\ell$ in $\OO_{\Ftilde}$ such that $v_{\frakl}(\frakN\frakD_{\Ktilde/\Ftilde}) >0$ and this prime $\frakl$ is unramified in $\Ktilde$.  Then for any $\calO_{\Ftilde}$-prime $\frakl'|\ell$, $\frakl'\neq \frakl$, the ideal  $\frakl'^{-1}\frakN\frakD_{\Ktilde/\Ftilde}$ is not integral and so $R_{\Ktilde/\Ftilde}(\frakl'^{-1}\frakN\frakD_{\Ktilde/\Ftilde}) = 0$.  Thus we have the simplified expression
				\[
					 B_{\frakN}(\ell) = 
						\begin{cases}
							0 & \textup{ if } \mathfrak{l} \textup{ splits in } 
								\widetilde{K}\\
						\frac12(v_{\mathfrak{l}}(\frakN)+1)R_{\Ktilde/\Ftilde}(\frakN \frakD_{\widetilde{K}/\widetilde{F}}\mathfrak{l}^{-1})f(\mathfrak{l}/\ell) & \textup{ otherwise}.
						\end{cases}
				\]

				To prove the local formula, first assume that $\frakl$ is not inert in $\Ktilde/\Ftilde$. By assumption, $\frakl$ is also unramified, so $\frakl$ must be split and $B_{\frakN}(\ell) = 0$.  If $\frakl$ is inert in $\Ktilde/\Ftilde$ and $v_{\frakl}(\frakN) = v_{\frakl}(\frakN \frakD_{\Ktilde/\Ftilde})$ is even, then $v_{\frakl}(\frakl^{-1}\frakN\frakD_{\Ktilde/\Ftilde})$ is odd. But since $\frakl$ is inert, the $\frakl$-valuation of $\Norm_{\Ktilde/\Ftilde}(\mathfrak{B})$ is even for any ideal $\mathfrak{B}$ of $\Ktilde.$  Thus $R(\frakl^{-1}\frakN\frakD_{\Ktilde/\Ftilde}) = 0 = B_{\frakN}(\ell)$.

				From now on, we may assume that $\frakl$ is inert in $\Ktilde$ and that $v_{\frakl}(\frakN)$ is odd.  Recall that by definition, we have $R_{\Ktilde/\Ftilde}(\frakl^{-1}\frakN\frakD_{\Ktilde/\Ftilde} ) = \#\{\fraka \subset \calO_{\Ktilde} \mid \Norm_{\Ktilde/\Ftilde}(\fraka) = \frakN \frakD_{\Ktilde/\Ftilde}\frakl^{-1}\}.$

				By the unique factorization of ideals in $\OO_{\Ktilde}$ and $\OO_{\Ftilde}$, 
				\[
					R_{\Ktilde/\Ftilde}(\frakN \frakD_{\Ktilde/\Ftilde}\frakl^{-1})
					= \prod_{p}\prod_{\pp|p}R_{\Ktilde/\Ftilde}(\frakp^{v_\pp(\frakN \frakD_{\Ktilde/\Ftilde}\frakl^{-1})}).
				\]
				Since the ideal $\frakN \frakD_{\Ktilde/\Ftilde}\frakl^{-1}$ is integral, we need only consider rational primes $p$ such that $p | N$.  Additionally, the factor at $\ell$ is equal to $1$, so 	
				\[
					R_{\Ktilde/\Ftilde}(\frakN \frakD_{\Ktilde/\Ftilde}\frakl^{-1})
					= \prod_{p|N, p\ne\ell}\prod_{\pp|p}R_{\Ktilde/\Ftilde}(\frakp^{v_\pp(\frakN \frakD_{\Ktilde/\Ftilde})}).
				\]
				Then it suffices to show that $\prod_{\pp|p}R_{\Ktilde/\Ftilde}(\frakp^{v_\pp(\frakN \frakD_{\Ktilde/\Ftilde})}) = \varepsilon_{\Ktilde/\Ftilde}(p, \frakN\frakD_{\Ktilde/\Ftilde})$ for $p\ne \ell$.

				Let $\pp | p$ be a prime ideal in $\calO_{\Ftilde}$. Let $v = v_{\pp}(\frakN \frakD_{\Ktilde/\Ftilde})$. If $\pp$ is inert in $\Ktilde$, then there is a unique prime $\frakP$ lying over $\pp$ and $\Norm_{\Ktilde/\Ftilde}(\frakP) = \pp^2$. Thus if $v$ is odd, there are no ideals in $\calO_{\Ktilde}$ whose relative norm has $\pp$-adic valuation $v$, and if $v$ is even, $R_{\Ktilde/\Ftilde}(\frakp^{v}) = 1$. Now suppose that $\pp$ splits in $\Ktilde$. Then we can write $\pp = \frakP_1 \frakP_2$ and $\Norm_{\Ktilde/\Ftilde}(\frakP_i) = \pp$ so the only ideals in $\OO_{\Ktilde}$ with relative norm equal to $\pp^v$ are of the form $\frakP_1^{n_1}\frakP_2^{n_2}$, where $n_1 + n_2 = v$ and $0 \leq n_1, n_2 \leq v$.  Thus $R_{\Ktilde/\Ftilde}(\frakp^{v_\pp(\frakN \frakD_{\Ktilde/\Ftilde})}) = v + 1$.  Finally, if $\pp \OO_{\Ktilde} = \frakP^2$ is ramified, then the only ideal in $\OO_{\Ktilde}$ with relative norm $\pp^v$ is $\frakP^v$, so $R_{\Ktilde/\Ftilde}(\frakp^{v_\pp(\frakN \frakD_{\Ktilde/\Ftilde})}) = 1$. We observe that this matches the expression for $\varepsilon_{\Ktilde/\Ftilde}(p, \frakN)$ exactly.
			\end{proof}
	
	\section{The Lauter-Viray formula}\label{sec:LV}
		In this section, we describe the formula for $(\CM(K).\textup{G}_1)_{\ell}$ proved by the third and last author~\cite{LV-Igusa}.  As in the Bruinier-Yang formula, some of the factors in this intersection formula are multiplicative. The key result of this section is Theorem~\ref{thm:LVLocal} where we show that the formula in~\cite{LV-Igusa} has an expression involving products of local factors. 

		\subsection{A simplified version of the Lauter-Viray formula}
		Throughout, we assume that $\OO_K$ is freely generated over $\OO_F$ and write $\eta$ for a generator, i.e., $\OO_K = \OO_F[\eta]$.  We define integers $\alpha_0,\alpha_1,\beta_0, \beta_1$ (depending on $\eta$) satisfying
		\[
			\Tr_{K/F}(\eta) = \alpha_0 + \alpha_1\frac{D + \sqrt{D}}{2},
			\textup{ and }
			\Norm_{K/F}(\eta) = \beta_0 + \beta_1\frac{D + \sqrt{D}}{2}.
		\]
		Let $\ell$ be a rational prime and let $\delta$ be a positive integer such that $D - 4\delta$ is a square.  We define
		\[
			c_K  := \alpha_0^2 + \alpha_0\alpha_1D 
				+ \alpha_1^2\frac{D^2 - D}{4} - 4\beta_0 - 2\beta_1D.
		\]
		For any integer $n$ such that $2D | (n + c_K)$ and $\frac{\delta^2\Dtilde - n^2}{4D}$ is a positive integer, we define 
		\begin{align*}
			t_u & := \alpha_1\delta,\\
			t_x & := \alpha_0 + \frac12(D - \sqrt{D - 4\delta})\alpha_1,\\
			d_u(n) & := (\alpha_1\delta)^2 + 4\frac{(n + c_K)\delta}{2D},\\
			d_x(n) & := (\alpha_0 + \frac12(D - \sqrt{D - 4\delta})\alpha_1)^2
				- 4\left(\beta_0 + \frac12(D - \sqrt{D - 4\delta})\beta_1 + \frac{n + c_K}{2D}\right),\\
			t_{xu^{\vee}}(n) & = \beta_1\delta 
				+ \sqrt{D - 4\delta}\frac{n + c_K}{2D}.
		\end{align*}
		The curious reader may refer to~\cite[\S2]{LV-Igusa} to see how these quantities arise.   

		\begin{theorem}~\cite[Thm. 2.10]{LV-Igusa}\label{thm:LVFormula}
			Assume that for every $\delta\in\Z_{>0}$ and $n\in\Z$ such that $D - 4\delta$ is a square, $2D|(n + c_K\delta)$ and $N := \frac{\delta^2\Dtilde - n^2}{4D} \in \Z_{>0}$, we have that $\ell$ does not divide both $\delta$ and $N$ and that $d_u(n)$ is a fundamental discriminant.  Then $(\CM(K).G_1)_{\ell}/(\log \ell)$ equals
			\begin{equation}\label{eq:LVFormula}
				\sum_{\underset{D - 4\delta = \square}{\delta\in\Z_{>0}}} \sum_{\substack{n \in\Z,\\2D|n + c_K\delta\\ \delta^2\Dtilde - n^2 \in {4D}\Z_{>0}}} \mu(n)\tilde\rho_{d_u(n)}(N)R(N\ell^{-1}), 
			\end{equation}
			where $R_{d}(A) = \#\{\frakb\subseteq \OO_d = \Z[\frac{d + \sqrt{d}}{2}] : \frakb \textup{ invertible }, \Norm(\frakb) = A\}$,
			\[
				\mu(n) = 
				\begin{cases}
					v_{\ell}(N) & \textup{if }\ell| \textup{gcd}(d_u(n),d_x(n)),\\
					\frac{v_{\ell}(N) + 1}2 & \textup{otherwise},
				\end{cases} \quad
				\tilde{\rho}_d(A) = 
				\begin{cases}
					0 & \textup{if }(d, -A)_p = -1\\
					&\textup{for some } p|d, p\ne\ell,\\
					2^{\#\{p : p|\textup{gcd}(d,A),p\ne\ell \}} & \textup{otherwise},
				\end{cases}
			\]
		and $(a,b)_p$ denotes the Hilbert symbol at $p$.
		\end{theorem}

		\subsection{A local interpretation of the Lauter-Viray formula}
			\begin{defn}
				Let $\OO_{d} = \Z[\frac{d + \sqrt{d}}{2}]$. Define 
				\[
					\varepsilon_{d}(p, A) = 
						\begin{cases}
							\frac12(1 + (-1)^{v_p(A)}) & 
								\textup{if }p\textup{ is inert in }\OO_{d}\\
							v_p(A) + 1 &
								\textup{if }p\textup{ is split in }\OO_{d}\\
							2 & 
								\textup{if }p|d\textup{ and }(d, - A)_p=1\\
							0 & \textup{otherwise}.
						\end{cases}\\
				\]
			\end{defn}
			\begin{thm}\label{thm:LVLocal}
			For any $\delta$, $n$, $N$ which arise as in Theorem~\ref{thm:LVFormula}, let $d_u = d_u(n)$ and $d_x = d_x(n)$. 
				Let $\ell$ be a prime that does not ramify in both $\OO_{d_u}$ and $\OO_{d_x}$.  Then
				\begin{equation}\label{eq:LVFormulaLocal}
					\mu(n)R_{d_u}(\ell^{-1}N)\tilde{\rho}_{d_u}(N) = 
					\begin{cases}
						\mu(n)\prod_{p|N, p\ne\ell}
						\varepsilon_{d_u}(p, N) & 
							\textup{if }\ell\textup{ is inert in }\OO_{d_u}
							\textup{ or }\OO_{d_x}\\
							&\textup{and }
							v_{\ell}(N) \equiv 1\bmod 2\\
						0 & \textup{otherwise}.
					\end{cases}
				\end{equation}
			\end{thm}
			\begin{proof}
				Recall that the Hilbert symbol $(a,b)_p$ remains unchanged when $a$ is multiplied by a norm from $\Q_p(\sqrt{b})$.  Thus since $d_xd_u = \Norm_{\Q_p(\sqrt{-N})}(t_xt_u - 2t_{xu^{\vee}} - 2\sqrt{-N})$~\cite[Eqn. 3.6]{LV-Igusa}, we have 
				\[
					\left(d_u, -N\right)_p = 
					\left(d_x, -N\right)_p
				\]
				for all primes $p$ (including $\infty$).

				Assume that $\ell$ is split in $\OO_{d_u}$ or $\OO_{d_x}$, or that $\ell$ is inert in $\OO_{d_u}$ or $\OO_{d_x}$ and $v_{\ell}(N)$ is even.  Then the right hand side of formula~\ref{eq:LVFormulaLocal} is zero.  If $\ell$ is split in $\OO_{d_u}$ or $\OO_{d_x}$, then the Hilbert symbol $\left(d_u, -N\right)_{\ell} = 1$ because either $d_u$ or $d_x$ is a square modulo $\ell$.  Recall that if $\Q_{\ell}(\sqrt{a})$ is a nontrivial unramified extension of $\Q_{\ell}$, then $(a,b)_{\ell} = 1$ if and only if $v_\ell(b)$ is even (\cite[Thm 1, p. 39]{Serre}).  Thus if $\ell$ is inert in $\OO_{d_u}$ or $\OO_{d_x}$ and $v_{\ell}(N)$ is even then $\left(d_u, -N\right)_{\ell} = 1$. By~\cite[Proof of Cor. 2.7]{LV-Igusa} $d_u$ is negative and so $\left(d_u, -N\right)_{\infty} = -1$. Therefore, by the product formula, there exists some prime $p\ne \ell$ such that $\left(d_u, -N\right)_{p} = -1$.  If $p$ is ramified in $\OO_{d_u}$, then this is exactly the condition to have $\tilde{\rho}_{d_u}(N) = 0$, so the left hand side of formula~\ref{eq:LVFormulaLocal} is also zero. If $p$ is unramified in $\OO_{d_u}$, then since $(d_u, -N)_p = -1$, $p$ must be inert in $\OO_{d_u}$ and $v_{p}(N)$ must be odd by the same argument as above.  In this case, there is no ideal in $\OO_{d_u}$ with norm $\ell^{-1}N$ and so $R_{d_u}(\ell^{-1}N) = 0$.

				Since, by assumption, $\ell$ does not ramify in both $\OO_{d_u}$ and $\OO_{d_x}$, the remaining case is when $\ell$ is inert in $\OO_{d_u}$ or $\OO_{d_x}$ and $v_{\ell}(N)$ is odd.  If $\ell$ is inert in $\OO_{d_x}$, then since $\ell$ divides $N = \frac14(d_xd_u - (d_xd_u - 2t)^2)$, $\ell$ is either inert or ramified in $\OO_{d_u}$.  In either case, $R_{d_u}(\ell^{2k}) = 1$ for any non-negative integer $k$, so
				\[
					R_{d_u}(\ell^{-1}N) 
					= \prod_{p | N}R_{d_u}(p^{v_p(\ell^{-1}N)})
					= \prod_{p | N, p\ne \ell}R_{d_u}(p^{v_p(N)}).
				\]
				Then, by the same argument as in the proof of Theorem~\ref{thm:BYLocal},
				\[
					R_{d_u}(\ell^{-1}N) =
						\prod_{p | N, p\ne \ell}
						\begin{cases}
							\frac12(1 + (-1)^{v_p(N)}) & 
								\textup{if }p\textup{ is inert in }\OO_{d_u},\\
							v_p(N) + 1 & 
								\textup{if }p\textup{ is split in }\OO_{d_u},\\
							1 & 
								\textup{if }p|d_u.
						\end{cases}
				\]
	Furthermore, it follows from the definition of $\tilde{\rho}_{d_u}$ that 			
				\[
					\tilde{\rho}_{d_u}(N) = \prod_{\substack{p | \textup{gcd}(N, d_u),\\ p\ne\ell}} 
						\begin{cases}
							2 & \textup{if }(d_u, -N)_p = 1,\\
							0 & \textup{if }(d_u, -N)_p = -1.
						\end{cases}
				\]
	From these two local expansions, it is clear that
				\[
					\mu(n)R_{d_u}(\ell^{-1}N)\tilde{\rho}_{d_u}(N) = 
					\mu(n)\prod_{p|N, p\ne \ell}\varepsilon_{d_u}(p, N)
				\]
				if $\ell$ is inert in $\OO_{d_u}$ or $\OO_{d_x}$ and $v_{\ell}(N)\equiv 1 \bmod 2$.  This completes the proof.
			\end{proof}
	
	\section{Relative integral bases}\label{sec:IntBasis}
	
		In the previous section (and hence throughout the paper), a number of quantities, such as $\alpha_0, \alpha_1, \beta_0, \beta_1$ and the others defined in terms of these, are expressed in a way that depends on the form of the integral basis $\{1, \eta\}$ for $\mathcal{O}_F$. In this section, we use a result of Spearman and Williams to determine the possible forms $\eta$ can take, thus narrowing down the possibilities for the other quantities given in Section 3.  Throughout, we let $A,B\in\frac12\Z$ be such that $A + B\sqrt{D}$ is squarefree in $\OO_F$ and such that $K= F(\sqrt{A + B\sqrt{D}}).$

		\begin{lemma}\label{lem:intbasis}
			Assume that $D$ and $\Dtilde$ are $1$ modulo $4$ and squarefree and that $\OO_K$ is freely generated over $\OO_F$.  Then a relative integral basis for $K$ over $F$ is $\{1, \eta\}$, where 
			\[
				\eta = \frac{1+\sqrt{A+B\sqrt{D}}}{2} 
				\quad\textup{ or }\quad
				\eta = \frac{2B + \sqrt{D} + 2\sqrt{A+B\sqrt{D}}}{4}.
			\]
			Furthermore, the latter case only occurs if $D \equiv 5\bmod 8$ and $A,B\in\frac12\Z\setminus\Z.$
		\end{lemma}
		\begin{proof}
	 		In~\cite{SW-RelativeIntegralBases}, Spearman and Williams give a necessary and sufficient condition for the existence of a relative integral basis for a quartic number field over a quadratic subfield.  In addition, in the cases where a relative integral basis exists, they give an explicit description of such an integral basis.  This lemma will follow almost immediately from their work.  

			Spearman and Williams use the classification of quartic number fields with a quadratic subfield given in an earlier paper of Huard, Spearman, and Williams~\cite{HSW-IntegralBases}; there are 51 cases with labels A1-A8, B1-B8, C1-C8, and D1-D27.  If $D\equiv 5 \bmod8$, then the field falls in cases C1 - C8 and if $D\equiv 1 \bmod 8$ then the field falls in cases D1-D27.  

			By~\cite[Thm 1]{SW-RelativeIntegralBases}, if $\Dtilde$ is squarefree, then $K$ falls into one of nine cases, only five of which have the property that $D\equiv 1 \bmod 4$.  These five cases are C2, C7, D3, D16, or D20.  If $D\equiv 1 \bmod 8$, then by~\cite[Thm 2]{SW-RelativeIntegralBases}, $\eta = \frac{1 + \sqrt{A + B\sqrt{D}}}{2}$.

			Now consider the case when $D\equiv 5 \bmod 8$.  By~\cite[Thm 2 and p.190]{SW-RelativeIntegralBases}, if $A,B\in\Z$, then $\eta = \frac{1 + \sqrt{A + B\sqrt{D}}}{2}$ and otherwise $A,B\in\frac12\Z\setminus\Z$ and $\eta = \frac{2B + \sqrt{D} + 	2\sqrt{A+B\sqrt{D}}}{4}$.
		\end{proof}

		We will use this lemma to give simplified expressions for 
	the quantities $d_u(n), d_x(n), t_{xu^{\vee}}(n)$ defined in~\S\ref{sec:LV}.  
		\begin{prop}\label{prop:SWquantities}
			Assume that $D$ and $\Dtilde$ are $1$ modulo $4$ and squarefree and that $\OO_K$ is freely generated over $\OO_F$.  Let $\delta\in\Z_{>0}$ be such that $D - 4\delta$ is a square and let $n\in\Z$ be such that $2D$ divides $(n + c_K\delta)$ and such that $\frac{\delta^2\Dtilde - n^2}{4D}$ is a positive integer.  Then
			\[
				\begin{array}{c}
				d_u(n) = \frac{\delta (2n+\delta 2A)}{D}, \;
				d_x(n) = A-B\sqrt{D-4\delta}-\frac{2n+\delta 2A}{D}, \;
				t_xt_u - 2t_{xu^\vee}(n) = B\delta - \frac{\sqrt{D-4\delta}(n+\delta A)}{D}.
				\end{array}
			\]
			Moreover, $c_K \equiv 1\bmod2$ and $2c_K \equiv 2A \bmod D$. 
		\end{prop}
		\begin{proof}

	Lemma~\ref{lem:intbasis} gives us two possible choices for $\eta$. In each case, we can explicitly give the values of $\alpha_i, \beta_i$. If $\eta = \frac{1+\sqrt{A+B\sqrt{D}}}{2}$, we have $\alpha_0 = 1, \alpha_1 = 0, \beta_0 = \frac{1-A+BD}{4}$, and $\beta_1 = -\frac{B}{2}$. If $\eta =\frac{2B + \sqrt{D} + 2\sqrt{A+B\sqrt{D}}}{4}$, we have $\alpha_0 =B- \frac{D}{2}, \alpha_1 = 1, \beta_0 = \frac{4B^2+D-4A}{16}$, and $\beta_1 = 0$. Using these values, we calculate $d_u(n)$, $d_x(n)$, and $t_xt_u - 2t_{xu^\vee}(n)$ and find that they have the desired expressions in both cases.

	To prove the congruence conditions, recall that 
			\[
				c_K =	\alpha_0^2 + \alpha_0\alpha_1D + \alpha_1^2\frac{D^2 - D}{4} - 4\beta_0 - 2\beta_1D.
			\]
	If $\alpha_0=1$ and $\alpha_1=0$, then $c_K \equiv 1 \bmod 2$. Otherwise, $D \equiv 5\pmod 8$, and so $\frac14(D^2 - D) \equiv 1\bmod 2$. Then, $c_K \equiv (\alpha_0^2 + \alpha_0\alpha_1 + \alpha_1^2) \bmod 2$. Since $\alpha_1 = 1$ in this case, we see that, regardless of the parity of $\alpha_0$, $c_K \equiv 1 \bmod 2$. 

	Calculating $c_K$ explicitly in each case, we see that $c_K$ is either equal to $A$ or $A-\frac{D}{2}$. In either case, it is clear that $2c_K \equiv 2A \bmod D$.\end{proof}
	
	\section{Equality of indices}\label{sec:EqualityIndices}

		The remainder of the paper will focus on proving, under slightly weaker assumptions than those in Theorems~\ref{thm:BYFormula} and~\ref{thm:LVFormula} and \emph{without} using Theorems~\ref{thm:LVFormula} and~\ref{thm:BYFormula}, that the expressions~\eqref{eq:BYFormula} and~\eqref{eq:LVFormula} agree. Precisely, we will show:
		\begin{theorem}\label{thm:Main}
			Assume that:
			\begin{itemize}
				\item $D$ is prime, and hence congruent to $1$ modulo $4$,
				\item $\Dtilde$ is squarefree and congruent to $1$ modulo $4$,
				\item $\Norm(\frakD_{\Ktilde/\Ftilde}) = D$, and
				\item for all $\delta\in\Z_{>0}$ such that $D - 4\delta$ is a square, and for all $n\in\Z$ such that $2D|n + c_K\delta$ and $4D|\delta^2\Dtilde - n^2$, $d_u(n)$ is a \emph{fundamental} discriminant, i.e., $d_u(n)$ is the discriminant of a quadratic field.
			\end{itemize}
			Then~\eqref{eq:BYFormula} and~\eqref{eq:LVFormula} are equal.
		\end{theorem}
		\begin{cor}\label{cor:BYconj}
			Retain the assumptions of Theorem~\ref{thm:Main}.  Then, the Bruinier-Yang conjectural formula for $(\CM(K).\textup{G}_1)_{\ell}$ holds.
		\end{cor}

		Both formula~\eqref{eq:BYFormula} and~\eqref{eq:LVFormula} involve summands indexed by two integers denoted $\delta$ and $n$.  The index $\delta$ ranges over the same quantities in both~\eqref{eq:BYFormula} and~\eqref{eq:LVFormula}.  While it is not obvious, the same statement is true for the index $n$. 

		\begin{prop}\label{prop:EquivCondOnn}
			Assume that $D$ and $\Dtilde$ are congruent to $1$ modulo $4$ and squarefree and that $\Norm(\frakD_{\Ktilde/\Ftilde}) = D$. Fix a positive integer $\delta$ such that $D - 4\delta$ is a square. Then for any $n\in\Z$,
			\[
				\delta^2\Dtilde - n^2 \in 4 D\Z \textup{ and  }
				n \equiv -\delta c_K \pmod{2D}
			\]
			if and only if
			\[
				\frac{n + \delta\sqrt{\Dtilde}}{2D}\in 
				\frakD_{\Ktilde/\Ftilde}^{-1}.
			\]
		\end{prop}
		\begin{remark}
			If we work with a different CM-type of $K$ so that $\Ktilde = \Ftilde(\sqrt{2A - 2\sqrt{A^2 - B^2D}})$, then the indices $n$ are in one-to-one correspondance, but not necessarily equal.  Indeed, the correspondence would be that\[
				\delta^2\Dtilde - n^2 \in 4 D\Z \textup{ and  }
				n \equiv -\delta c_K \pmod{2D}
			\]
			if and only if
			\[
				\frac{-n + \delta\sqrt{\Dtilde}}{2D}\in 
				\frakD_{\Ktilde/\Ftilde}^{-1}.
			\]
		\end{remark}
		\begin{proof}
			We will need the factorization of $\langle p\rangle$ in $\OO_{\Ftilde}$ for any $p|D$, so we present this first. Recall that $A$ and $B$ are chosen to be in $\frac12\Z$ such that $\Dtilde = A^2 - B^2 D$.  Since $2A$ is a solution of $X^2 - 4\Dtilde \bmod D$ and $2\nmid D$, for any $p|D$, we can factor $\langle p \rangle $ in $\OO_{\Ftilde}$ as $\pp_1\pp_2$ where $\pp_1 = (2A - 2\sqrt{\Dtilde}, p)$ and $\pp_2 = (2A + 2\sqrt{\Dtilde}, p)$.  Note that $\pp_1 = \pp_2$ if and only if $p|\Dtilde$ as well as $D$.
	The norm of $\frakD_{\Ktilde/\Ftilde}$ is equal to $D$ and since $p|D$  one of $\pp_1$ or $\pp_2$ must ramify in $\Ktilde/\Ftilde$.  Since $D$ is squarefree, at most one of $\pp_1$ or $\pp_2$ ramifies and $\frakD_{\Ktilde/\Ftilde}$ has $\pp_i$-adic valuation at most 
	$1$.  Recall that $\Ktilde = \Q(\sqrt{2A+2\sqrt{\Dtilde}})$, thus $\pp_2|\frakD_{\Ktilde/\Ftilde}$.

			First assume that $\delta^2\Dtilde - n^2 \in 4 D\Z$ and that $n \equiv -\delta c_K \pmod{2D}$; we will show that $\frac{n + \delta\sqrt{\Dtilde}}{2D} \in \frakD_{\Ktilde/\Ftilde}^{-1}$. Since $\delta^2\Dtilde - n^2$ is divisible by $4$, $n$ must be congruent to $\delta$ modulo $2$ and thus $\frac{n + \delta \sqrt{\Dtilde}}{2}$ is integral.  Further, since $\frakD_{\Ktilde/\Ftilde}$ is integral, so is $(\frac{n + \delta \sqrt{\Dtilde}}{2})\frakD_{\Ktilde/\Ftilde}$.  To prove that $\frac{n + \delta\sqrt{\Dtilde}}{2D}\in\frakD_{\Ktilde/\Ftilde}^{-1}$, we will show that every prime lying over $p$ for $p|D$ either divides $\frac{n + \delta\sqrt{\Dtilde}}{2}$ or $\frakD_{\Ktilde/\Ftilde}.$  In addition, if $p$ also divides $\Dtilde$, then we will show that the unique prime $\pp$ lying over $p$ divides $\frakD_{\Ktilde/\Ftilde}$ and $\frac{n + \delta\sqrt{\Dtilde}}{2}$. Note that we have $p > 2$ since $D$ and $\Dtilde$ are assumed to be $1$ modulo $4$.

		By assumption, $2D | (n + c_K \delta)$ and by Proposition~\ref{prop:SWquantities} we have $2c_K\equiv 2A\bmod{D}$, so
				\[
					2n + 2\delta \sqrt{\Dtilde} \equiv 
					-2\delta A + 2\delta \sqrt{\Dtilde}
				 	\equiv 0  \bmod \pp_1
				\]
		and thus $v_{\pp_1}(\frac{n + \delta\sqrt\Dtilde}{2D}) >0$.  We have already seen that $\pp_2|\frakD_{\Ktilde/\Ftilde}$. Therefore, $\frac{n + \delta\sqrt\Dtilde}{2D}\in\frakD_{\Ktilde/\Ftilde}^{-1}.$

			Now we prove the reverse direction.  Assume that $\frac{n + \delta\sqrt{\Dtilde}}{2D} \in \frakD^{-1}_{\Ktilde/\Ftilde}$. Taking the absolute norm, we have
			\begin{align*}
			N_{\Ftilde/\Q} \left(\left(\frac{n + \delta \sqrt{\Dtilde}}{2D}\right) 
			\frakD_{\Ftilde/\Ktilde}\right) &= 
			\frac{n^2 - \delta^2 \Dtilde}{4D^2} \cdot 
			N_{\Ftilde/\Q}(\frakD_{\Ktilde/\Ftilde})\\
			& = \frac{n^2 - \delta^2 \Dtilde}{4D} \cdot 
			\frac{N_{\Ftilde/\Q}(\frakD_{\Ktilde/\Ftilde})}{D} \in \Z.
			\end{align*}
			Since $N_{\Ftilde/\Q}(\frakD_{\Ktilde/\Ftilde}) = D$, we have $\delta^2 \Dtilde - n^2 \in 4 D \Z$.

			To prove the congruence condition, we use the fact that $\pp_2 | \frakD_{\Ktilde/\Ftilde}$.  Since $\Norm(\frakD_{\Ktilde/\Ftilde})$ is squarefree and $p$ divides $\frac{n + \delta\sqrt{\Dtilde}}{2}\frakD_{\Ktilde/\Ftilde}$, this implies that $\pp_1 | (n + \delta \sqrt{\Dtilde})/2$.  Since $2n + 2\delta\sqrt{\Dtilde} \equiv 2n + 2\delta A \pmod{\pp_1}$, the integer $2n + 2\delta A$ is contained in $\pp_1$ and hence is $0$ modulo $p$, for all $p|D$.  This implies that $n + \delta A \equiv n + c_K\delta\equiv0 \bmod D$. We have already shown that $\delta^2\Dtilde - n^2 \in 4\Z$, which implies that $n \equiv \delta \pmod 2$. Finally, by Proposition~\ref{prop:SWquantities}, $c_K \equiv 1 \pmod 2$. Thus, $n \equiv \delta c_K \pmod 2$, and the proof is complete.
		\end{proof}

	\section{Equality of summands}\label{sec:EqualitySummands}

		By the results of the previous section, both formula~\eqref{eq:BYFormula} and formula~\eqref{eq:LVFormula} sum over the same values $\delta$ and $n$.  Thus, to prove that the formulas agree, it suffices to show that for a fixed $\delta$ and $n$, the corresponding summands of formula~\eqref{eq:BYFormula} and~\eqref{eq:LVFormula} are equal. The goal of the present section is to prove this equality.

		Throughout, we work with a fixed positive integer $\delta$ and a fixed integer $n$ such that
		\[
			D - 4\delta = \square, \quad n + c_K\delta \equiv 0 \bmod{2D}, \quad
			\textup{ and }N := \frac{\delta^2\Dtilde - n^2}{4D}\in\Z_{>0}.
		\]
		For simplicity, we write $d_u := d_u(n)$ and $d_x := d_x(n)$.  We let $\OO_{d_u}$ and $\OO_{d_x}$ denote the quadratic imaginary orders of discriminant $d_u$ and $d_x$ respectively.  

		Precisely, in this section we prove:
		\begin{theorem}\label{thm:SummandTheorem}
			Retain the assumptions from Theorem~\ref{thm:Main}.  Then for any prime $\ell$,
			\begin{equation}\label{eq:BYLVSummandEquality}
				\mu(n)R_{d_u}(\ell^{-1}N)\tilde{\rho}_{d_u}(N)
				 = \sum_{\frakl|\ell}
				\begin{cases}
					0 & \textup{if }\frakl\textup{ is split in }\Ktilde\\
					\frac{v_{\frakl}(\frakN) + 1}{2}\cdot f(\mathfrak{l}/\ell) \cdot R_{\Ktilde/\Ftilde}
					(\frakl^{-1}\frakN\frakD_{\Ktilde/\Ftilde}) & 
					\textup{otherwise},
				\end{cases}
			\end{equation}
			where $\frakN = (\frac{n+\delta\sqrt{\Dtilde}}{2D})$.
		\end{theorem}

		In~\S\ref{subsec:Reduction}, we prove restrictions on the prime divisors of $N$.  These restrictions will prove useful in later sections, and they also allow us to give a simplified formula for $\mu(n)$. In~\S\ref{subsec:ValuationsAndSplitting}, we consider the splitting behavior in $\OO_{d_u}$ and $\OO_{d_x}$ of primes $p$ dividing $N$ and relate it to the splitting behavior in $\Ktilde$ of primes $\pp$ dividing $\frakN.$  We use this  in~\S\ref{subsec:ComparingLocalFactors} to show that for each prime $p\ne \ell$, the local factor at $p$ in formula~\eqref{eq:BYFormulaLocal} agrees with the local factor at $p$ in formula~\eqref{eq:LVFormulaLocal}.  Finally, in~\S\ref{subsec:ProofOfSummandTheorem}, we explain how these ingredients come together to prove Theorem~\ref{thm:SummandTheorem}.

		\subsection{Reduction steps}\label{subsec:Reduction}

			\begin{lemma}\label{lem:pDoesntDivideDelta}
				Retain the assumptions from Theorem~\ref{thm:Main}.  Then $\delta$ and $N = \frac{\delta^2\Dtilde - n^2}{4D}$ are relatively prime. 
			\end{lemma}

			\begin{proof}
				First suppose that $p$ is an odd prime. If $p$ divides both  $\delta$ and $\frac{\delta^2\Dtilde - n^2}{4D}$, $p$ must also divide $n$. Since $D$ is prime and $p \leq \delta < D$, $p$ cannot divide $D$, and so $p^2$ must divide $d_u(n) = \frac{\delta(2n+\delta2A)}{D}$. This violates the hypothesis that $d_u(n)$ is the discriminant of an imaginary quadratic field.

				Now let $p = 2$ and assume that $p|N$ and $p|\delta$. Then since $D - 4\delta$ is a square, $D$ must be congruent to $1$ modulo $8$.  Since $\Dtilde = A^2 - B^2 D$ is $1$ modulo $4$, $A$ and $B$ must be integers and $A$ must be odd.  By assumption, $8 | \delta^2\Dtilde - n^2$ and $2|\delta$, so $n \equiv \delta\equiv \delta A\bmod 4$.  Thus $d_u(n) = 2\delta(n + \delta A)/D$ is $0$ modulo $16$, which gives a contradiction. 
			\end{proof}

			\begin{prop}\label{prop:DuAndDxRelativelyPrime}
				Assume that $\Dtilde$ is squarefree and fix a prime $p$ that does not divide $\delta$.  If $p|N$, then $p$ cannot divide both $d_u(n)$ and $d_x(n)$. 
			\end{prop} 

			\begin{proof}
				Suppose $p$ divides both $d_u(n)$ and $d_x(n)$. Recall that we have 			
				\begin{equation}\label{eqn:defn-of-N}
					\delta^2\Dtilde - n^2 = D\left(d_x(n)d_u(n) - (t_xt_u - 2t_{xu^\vee}(n))^2\right).
				\end{equation}
				If $4pD$ divides the left hand side of this equation, then $p$ must also divide $(t_xt_u - 2t_{xu^\vee}(n))$. Using the formulations for this quantity, $d_u(n)$, and $d_x(n)$ given in Proposition~\ref{prop:SWquantities}, we see that if $p | d_u(x)$, then $p | \frac{2n+\delta 2A}{D}$. If $p | \frac{n+\delta A}{D}$ and $p |(t_xt_u - 2t_{xu^\vee}(n))$, then $p | 2B$. But, if $p$ divides all of these quantities, by considering the expression for $d_x(n)$ in Proposition~\ref{prop:SWquantities}, we see that $p$ must also divide $2A$. Furthermore, if $p = 2$, then this argument can be strengthened to show that $A$ and $B$ are even integers.  However, $A$ and $B$ must be relatively prime, because $\Dtilde = A^2-B^2D$ is assumed to be squarefree. Thus, $p$ cannot divide both $d_u(n)$ and $d_x(n)$.
			\end{proof}

			\begin{cor}\label{corr:mu}
				Retain the assumptions from Theorem~\ref{thm:Main}.  If $\ell|N$, then $\mu(n) = \frac12(v_{\ell}(N) + 1).$
			\end{cor}

		\subsection{Comparing valuations and splitting behavior}
		\label{subsec:ValuationsAndSplitting}

			\begin{lemma} \label{lem:Valuations}
				Retain the assumptions from Theorem~\ref{thm:Main}. Let $n\in\Z$ be such that $2D|(n + c_K\delta)$ and that $\frac{\delta^2\Dtilde - n^2}{4D}\in\Z_{>0}$.  Let $p$ be a prime that divides $\frac{\delta^2\Dtilde - n^2}{4D}$. Then there is a unique prime $\frakp\in\OO_{\Ftilde}$ lying over $p$ such that $v_{\frakp}\left(\frac{n + \delta\sqrt{\Dtilde}}{2D}\frakD_{\Ktilde/\Ftilde}\right)$ is positive.  This prime $\frakp$ is unramified in $\Ktilde$, $f(\frakp/p) = 1$, and we have
				\[
					v_{p}\left(\frac{\delta^2\Dtilde - n^2}{4D}\right) = 
					v_{\frakp}\left(\frac{n + \delta\sqrt{\Dtilde}}{2D}\right) = 
					v_{\frakp}\left(\frac{n + \delta\sqrt{\Dtilde}}{2D}
					\frakD_{\Ktilde/\Ftilde}\right).
				\]
			\end{lemma}
			\begin{remark}
				This lemma shows that the assumptions in Theorem~\ref{thm:Main} imply the assumptions in Theorem~\ref{thm:BYLocal}.
			\end{remark}
			\begin{proof}
				By Lemma~\ref{lem:pDoesntDivideDelta}, $p\nmid\delta$, so there is at most one prime in $\OO_{\Ftilde}$ lying over $p$ that divides $\frac{n + \delta\sqrt{\Dtilde}}{2}$ and this prime has inertial degree $1$ over $p$.  First consider the case when $p\nmid D$.  Since $\Norm(\mathfrak{D}_{\Ktilde/\Ftilde}) = D$, we have that for all $\frakp|p$, $\frakp$ is unramified in $\Ktilde$ and $v_{\frakp}\left(\frac{n + \delta\sqrt{\Dtilde}}{2D}\mathfrak{D}_{\Ktilde/\Ftilde}\right) = v_{\frakp}\left(\frac{n + \delta\sqrt{\Dtilde}}{2}\right)$.  As 
				\begin{equation}\label{eq:ValuationRelation}
					v_{p}\left(\frac{\delta^2\Dtilde - n^2}{4D}\right) = 
					\sum_{\frakp|p } v_{\frakp}
					\left(\frac{n + \delta\sqrt{\Dtilde}}{2D}
					\frakD_{\Ktilde/\Ftilde}\right),
				\end{equation}
				this completes the proof.

				Now consider the case when $p | D$.  If $p$ is ramified in $\Ftilde$, then $p|\Dtilde$.  However, this contradicts the assumption that $\frac{\delta^2\Dtilde - n^2}{4D}\in p\Z_{>0}$ because $\Dtilde$ is squarefree and $p\nmid\delta$.  Thus $p$ is split in $\Ftilde$.  Let $\frakp_1$ and $\frakp_2$ denote the two primes lying over $p$.  Since $\Norm(\frakD_{\Ktilde/\Ftilde}) = D$ and $D$ is a prime,  there is at most one prime lying over $p$ that divides $\frakD_{\Ktilde/\Ftilde}$; we may assume that this prime is $\frakp_2$. Hence $v_{\frakp_1}\left(\frac{n + \delta\sqrt{\Dtilde}}{2D}\frakD_{\Ktilde/\Ftilde}\right) = v_{\frakp_1}\left(\frac{n + \delta\sqrt{\Dtilde}}{2}\right) - 1.$ By the assumption on $n$ and Proposition \ref{prop:EquivCondOnn}, $\frac{n + \delta\sqrt{\Dtilde}}{2D}\frakD_{\Ktilde/\Ftilde}$ is integral, and thus $v_{\frakp_1}\left(\frac{n + \delta\sqrt{\Dtilde}}{2}\right) > 0.$  This in turn implies that $\frac12(n + \delta\sqrt{\Dtilde})$ is a $\frakp_2$-adic unit.  Combining this with~\eqref{eq:ValuationRelation}, we see that
		$v_{p}\left(\frac{\delta^2\Dtilde - n^2}{4D}\right) = 
		v_{\frakp_1}\left(\frac{n + \delta\sqrt{\Dtilde}}{2D}\right)$ and $\frakp_1\nmid\frakD_{\Ktilde/\Ftilde}$ as desired.
			\end{proof}

			\begin{prop}\label{prop:Splitting}
				Retain the assumptions from Theorem~\ref{thm:Main}.
				Fix a prime $p$ that divides $\frac{\delta^2\Dtilde - n^2}{4D}$ and let $\frakp|p$ be the unique prime given in Lemma~\ref{lem:Valuations}.  The prime ideal $\frakp$ splits in $\Ktilde$ if and only if $p$ splits in at least one of $\OO_{d_x}$ or $\OO_{d_u}$.  Similarly, $\frakp$ is inert in $\Ktilde$ if and only if $p$ is inert in at least one of $\OO_{d_x}$ or $\OO_{d_u}$.
			\end{prop}
			\begin{proof}
				By Lemma~\ref{lem:pDoesntDivideDelta} and Proposition~\ref{prop:DuAndDxRelativelyPrime}, $p$ does not ramify in both $\OO_{d_u}$ and $\OO_{d_x}$.  Therefore, if $p$ is not split in either $\OO_{d_x}$ or $\OO_{d_u}$, then $p$ is inert in at least one of $\OO_{d_u}$ and $\OO_{d_x}$.  Thus, the second claim of the lemma follows from the first claim.

				As noted above, $\frac{\delta^2\Dtilde - n^2}{4D} = \frac{d_ud_x - (t_xt_u - 2t_{xu^{\vee}})^2}{4}$.  Since $4Dp|(\delta^2\Dtilde - n^2)$, the product $d_u d_x$ is congruent to a square modulo $p$.  Therefore, if $p$ is split in one of $\OO_{d_x}$ or $\OO_{d_u}$, then $p$ cannot be inert in the other order.  If $p > 2$, the proof breaks into cases depending on whether or not $p$ ramifies in $\OO_{d_u}$. Recall, $d_u=\frac{2\delta (n+\delta A)}{D}$. Assume that $p|d_u$ and $p > 2$. Then, since $p\nmid\delta$ (Lemma~\ref{lem:pDoesntDivideDelta}), $2n + 2A\delta$ and $n + \delta\sqrt{\Dtilde}$ both have $\frakp$-adic valuation strictly greater than $v_{p}(D)$, and hence so does $2A - 2\sqrt{\Dtilde}$.  This in turn implies that $p$ divides $2B$ and so $d_x =A-B\sqrt{D-4\delta}-\frac{2(n+\delta A)}{D} \equiv A\bmod p$.  Recall that $\sqrt{2A + 2\sqrt{\Dtilde}}$ generates the extension $\Ktilde/\Ftilde$.  Consider the product
				\begin{equation}\label{eq:Congruence1}
					(2A + 2\sqrt{\Dtilde})d_x \equiv 4A\cdot A\pmod\frakp.
				\end{equation}
				Since $\Dtilde$ is squarefree, $4A^2$ is a nonzero square modulo $\frakp$. Then~\eqref{eq:Congruence1} implies that $2A + 2\sqrt\Dtilde$ and $d_x$ are nonzero modulo $\frakp$ and that $2A + 2\sqrt{\Dtilde}$ is a square modulo $\frakp$ if and only if $d_x$ is a square modulo $p$.  This shows that $\frakp$ splits in $\Ktilde$ if and only if $p$ splits in $\OO_{d_x}$. 

				Suppose that $p \nmid d_u$ and $p > 2$.  Then by the argument above, $2A - 2\sqrt{\Dtilde}$ is a $\frakp$-adic unit.  Thus, if $p|2B$, we must have that $\frakp | (2A + 2\sqrt{\Dtilde})$.  We will show that $(2A + 2\sqrt{\Dtilde})d_u$ is congruent to a nonzero square modulo $\pp^{2v_p(2B) + 1}$.  This will show that $\frakp$ splits in $\Ktilde$ if and only if $p$ splits in $\OO_{d_u}$.

				Since $2A - 2\sqrt{\Dtilde}$ is a $\pp$-adic unit, $v_{\pp}(2A + 2\sqrt{\Dtilde}) = 2v_p(2B).$  By assumption, we also have that $v_{\pp}(\frac{n + \delta\sqrt{\Dtilde}}{2D})$ is positive, so
				\[
					(2A + 2\sqrt\Dtilde)\delta^2\frac{2n/\delta + 2\sqrt{\Dtilde}}{D} \in \pp^{2v_p(2B) + 1}.
				\]
				From this, we see that
				\[
					(2A + 2\sqrt{\Dtilde})d_u\equiv \frac{\delta^2}{D}(2A + 2\sqrt{\Dtilde})(2A - 2\sqrt{\Dtilde}) \equiv \delta^2(2B)^2\bmod{\pp^{2v_p(2B) + 1}}.
				\]
				By Lemma~\ref{lem:pDoesntDivideDelta}, $p\nmid\delta$, so we obtain our result.

				Henceforth, we assume that $p = 2$.  Suppose that $A$ and $B$ are half-integers, i.e., that $2A$ and $2B$ are odd integers.  Then $\frac{2n + 2\delta\sqrt{\Dtilde}}{D}$ is zero modulo $\pp^3$, so $d_u = \delta^2\frac{2A + 2n/\delta}{D}\equiv \delta^2\frac{2A - 2\sqrt\Dtilde}{D}\bmod{\pp^3}.$  Thus
				\[
					(2A + 2\sqrt{\Dtilde})d_u\equiv \delta^2(2B)^2\bmod{\pp^3}.
				\]
				Since $\delta2B$ is odd, this shows that $\frakp$ splits in $\Ktilde$ if and only if $p$ splits in $\OO_{d_u}$.

				If $A$ and $B$ are integers, then $d_u$ is necessarily divisible by $2$ and $\frac{A + \sqrt{\Dtilde}}{2}\in\OO_{\Ftilde}$.  Suppose that $d_u \equiv 8 \bmod{16}$ so $\frac{A + n/\delta}{2} \equiv 0 \bmod 2$.  Then $\frac{A - \sqrt{\Dtilde}}{2} = \frac{A + n/\delta}{2} - \frac{n/\delta + \sqrt{\Dtilde}}{2}$ is $0$ modulo $\pp$.  The discriminants $D$ and $\Dtilde = A^2 - B^2D$ are $1$ modulo $4$, so $A$ must be odd and $B$ must be even.  Since $\frac{A - \sqrt\Dtilde}{2}\in \pp$, $A^2\equiv\Dtilde\bmod{8}$ and so $B$ must be divisible by $4$. Thus $\frac{A + \sqrt{\Dtilde}}{2}$ is a $\pp$-adic unit.  Consider 
				\[
					\frac{A + \sqrt{\Dtilde}}{2} - d_x \equiv \frac{-A + \sqrt{\Dtilde}}{2} + B\sqrt{D - 4\delta}\bmod{\pp^3}.
				\]
				Since $\Norm(\frac{-A + \sqrt\Dtilde}{2}) = B^2D/4$ and $D$ is $1$ modulo $4$, $v_{\pp}(\frac{-A + \sqrt\Dtilde}{2}) = v_p(B^2/4)$.  If $v_p(B) \geq 3$, then $\frac{A + \sqrt{\Dtilde}}{2} \equiv d_x\bmod{\pp^3}$.  If $v_p(B) = 2$, then $v_{\pp}(\frac{-A + \sqrt\Dtilde}{2})$ is also $2$, so the sum $\frac{-A + \sqrt{\Dtilde}}{2} + B\sqrt{D - 4\delta}$ has $\pp$-adic valuation at least $3$.  Therefore, in all cases,  $\frac{A + \sqrt{\Dtilde}}{2} \equiv d_x\bmod{\pp^3}$ and so $\frakp$ splits in $\Ktilde$ if and only if $p$ splits in $\OO_{d_u}$.

				Finally we suppose that $d_u \equiv 4 \bmod 8$.  By assumption, $d_u$ is fundamental which implies that $\frac{A + n/\delta}{2}$ is $3$ modulo $4$.  Since $d_x$ is a quadratic discriminant and congruent to $A$ modulo $4$, $A$ must be congruent to $1$ modulo $4$, so $\frac{A - n/\delta}{2}$ is $2$ modulo $4$. Both $\frac{A - n/\delta}{2}$ and $\frac{n/\delta + \sqrt{\Dtilde}}{2}$ are $\pp$-adic uniformizers, hence their sum and difference both have $\pp$-adic valuation at least $2$.  Moreover, at most one of the sum and difference have $\pp$-adic valuation exactly equal to $2$.  In particular, $\frac{A + \sqrt{\Dtilde}}{2}$ has positive valuation, so $v_p(B^2/4)$ must be positive (so $B$ is $0$ modulo $4$) and $v_p(B^2/4) = v_{\pp}(\frac{A + \sqrt\Dtilde}{2}).$  From this, we can see that $\pp$ splits in $\Ktilde$ if 
				\begin{equation}\label{eq:GeneratorKtilde}
					\left(\frac{A + \sqrt\Dtilde}{2}\right)
					\left(\frac{A - \sqrt\Dtilde}{2}\right)^24B^{-2} = 
					\left(\frac{A - \sqrt\Dtilde}{2}\right)D
				\end{equation}
				is a square modulo $\pp^3$, and that $\pp$ is inert in $\Ktilde$ if~\eqref{eq:GeneratorKtilde} is a non-square modulo $\pp^3$.

				If $B \equiv 4 \bmod 8$, then $v_{\pp}(\frac{A + \sqrt{\Dtilde}}{2}) = 2$ and $v_{\pp}(\frac{-A + 2n/\delta+ \sqrt{\Dtilde}}{2})\geq 3$.  Therefore
				\[
					\frac{A - \sqrt\Dtilde}{2} \equiv n/\delta \equiv A + 4
					\bmod \pp^3, \quad \textup{ and }
					d_x = A - B\sqrt{D - 4\delta} - d_u/\delta \equiv A + 4 + 4 
					\equiv A \bmod\pp^3,
				\]
				so $\left(\frac{A - \sqrt\Dtilde}{2}\right)D\cdot d_x$ is equivalent to $D(A^2 + 4A)$ modulo $\pp^3$.  Since $D - 4\delta$ is a square and $\delta$ is odd, $D$ must be $5$ modulo $8$.  Thus $D(A^2 + 4A)\equiv 5(1 + 4) \equiv 1 \bmod 8$, so $\frakp$ splits in $\Ktilde$ if and only if $p$ splits in $\OO_{d_x}$.  If $B \equiv 0\bmod 8$, then $v_{\pp}(\frac{A + \sqrt{\Dtilde}}{2}) \geq 3$ and so $\frac{A - \sqrt\Dtilde}{2} \equiv A \bmod\pp^3$.  We also have $d_x \equiv A + 4\bmod 8$.  Thus, as above, $\left(\frac{A - \sqrt\Dtilde}{2}\right)D\cdot d_x\equiv D(A^2 + 4A) \equiv 1\bmod\pp^3$.  This completes the proof.
			\end{proof}

		\subsection{Comparing $\varepsilon_{d_u}$ and 
		$\varepsilon_{\Ktilde/\Ftilde}$}\label{subsec:ComparingLocalFactors}

			\begin{prop}\label{prop:LocalFactorEquality}
				Retain the assumptions from Theorem~\ref{thm:Main}.
				Let $n\in\Z$ be such that $2D|(n + c_K\delta)$ and that $\frac{\delta^2\Dtilde - n^2}{4D}\in\ell\Z_{>0}$.  Fix a prime $p\neq\ell$ that divides $N := \frac{\delta^2\Dtilde - n^2}{4D}$.  Then
				\[
					\varepsilon_{d_u}(p,N) = \varepsilon_{\Ktilde/\Ftilde}(p, \frakN\frakD_{\Ktilde/\Ftilde}).
				\]
			\end{prop}
			\begin{proof}
				By Lemma~\ref{lem:Valuations}, there is a unique prime $\frakp$ lying over $p$ such that $v_{\frakp}(\frakN\frakD_{\Ktilde/\Ftilde})$ is positive.  Thus, 
				\[
					\varepsilon_{\Ktilde/\Ftilde}(p, \frakN\frakD_{\Ktilde/\Ftilde}) = 
						\begin{cases}
							\frac12(1 + (-1)^{v_{\pp}(\frakN\frakD_{\Ktilde/\Ftilde})}) & 
							\textup{if }\pp\textup{ is inert in }\Ktilde\\
							v_{\pp}(\frakN\frakD_{\Ktilde/\Ftilde}) + 1& 
							\textup{if }\pp\textup{ is split in }\Ktilde\\
							1 & \textup{otherwise}.
						\end{cases}
				\]

				Assume that $\frakp$ is inert in $\Ktilde.$  Then, by Proposition~\ref{prop:Splitting}, $p$ is inert in at least one of $\OO_{d_u}$ or $\OO_{d_x}$.  If $p$ is inert in $\OO_{d_u}$, then
				\[
					\varepsilon_{d_u}(p, N) = \frac12(1 + (-1)^{v_{p}( N)})
					= \frac12(1 + (-1)^{v_{\pp}( \frakN\frakD_{\Ktilde/\Ftilde})})
					= \varepsilon_{\Ktilde/\Ftilde}(p, \frakN\frakD_{\Ktilde/\Ftilde}),
				\]
				as desired.  (The middle equality follows from Lemma~\ref{lem:Valuations}.)  If $p$ is not inert in $\OO_{d_u}$, then $p$ must be inert in $\OO_{d_x}$.  The equality $\frac{\delta^2\Dtilde - n^2}{4D} = \frac{d_ud_x - (t_xt_u - 2t_{xu^{\vee}})^2}{4}$, shows that $d_ud_x$ is congruent to a square modulo $p$.  Thus $p$ is ramified in $\OO_{d_u}$.  In addition, by Lemma~\ref{lem:Valuations} and since $p^2$ does not divide $d_u$, 
				\[
					v_p\left(\frac{\delta^2\Dtilde - n^2}{4D}\right) = v_{\pp}(\frakN\frakD_{\Ktilde/\Ftilde}) = 1.
				\]  
		Since $d_x$ is not a square modulo $p$ and the $p$-valuation of $N$ is odd, it follows again from~\cite[Ch III, Thm 1]{Serre} that
				\[
					\left(d_u, \frac{n^2 - \delta^2\Dtilde}{4D}\right)_p = 
					\left(d_x, \frac{n^2 - \delta^2\Dtilde}{4D}\right)_p = -1.
				\]
				Thus $\varepsilon_{d_u}(p, \ell^{-1}N) = 0 = \frac12(1 + (-1)^{v_{\pp}(\frakN\frakD_{\Ktilde/\Ftilde})}) = \varepsilon_{\Ktilde/\Ftilde}(p, \frakN\frakD_{\Ktilde/\Ftilde})$.

				If $\frakp$ is not inert in $\Ktilde$, then, by Lemma~\ref{lem:Valuations}, $\frakp$ is split in $\Ktilde$.  By Proposition~\ref{prop:Splitting}, this implies that $p$ is split in at least one of $\OO_{d_x}$ or $\OO_{d_u}$.  If $p$ is split in $\OO_{d_u}$, then
				\[
					\varepsilon_{d_u}(p, N) = v_{p}(N) + 1
							= v_{\pp}(\frakN\frakD_{\Ktilde/\Ftilde}) + 1
							= \varepsilon_{\Ktilde/\Ftilde}(p, \frakN\frakD_{\Ktilde/\Ftilde}).
				\]
				  If $p$ is not split in $\OO_{d_u}$, then $p$ is split in $\OO_{d_x}$, and the same arguments as above show that $p$ is ramified in $\OO_{d_u}$, $v_p(\frac{\delta^2\Dtilde - n^2}{4D}) = 1$. Furthermore, 
				\[
					\left(d_u, \frac{n^2 - \delta^2\Dtilde}{4D}\right)_p = 
					\left(d_x, \frac{n^2 - \delta^2\Dtilde}{4D}\right)_p = 1
				\]
				and so $\varepsilon_{d_u}(p, N) = 2 = v_{p}(N) + 1
				= v_{\pp}(\frakN\frakD_{\Ktilde/\Ftilde}) + 1 = \varepsilon_{\Ktilde/\Ftilde}(p, \frakN\frakD_{\Ktilde/\Ftilde})$.  This completes the proof.
			\end{proof}

		\subsection{Proof of Theorem~\ref{thm:SummandTheorem}}
		\label{subsec:ProofOfSummandTheorem}

			Let $\frakl$ denote the prime lying over $\ell$ such that $v_{\frakl}(\frac{n + \delta\sqrt{\Dtilde}}{2D}\frakD_{\Ktilde/\Ftilde})$ is positive; this is unique by Lemma~\ref{lem:Valuations}.  If $\frakl$ is split in $\Ktilde$, then by Theorem~\ref{thm:BYLocal} the right-hand side of~\eqref{eq:BYLVSummandEquality} is zero.  Additionally, by Proposition~\ref{prop:Splitting}, if $\frakl$ is split in $\Ktilde$, then $\ell$ is split in $\OO_{d_u}$ or $\OO_{d_x}$.  By Theorem~\ref{thm:LVLocal}, this implies that the left-hand side of~\eqref{eq:BYLVSummandEquality} is zero.

			Since, by Lemma~\ref{lem:Valuations}, $\frakl$ is unramified in $\Ktilde$, we are left to consider the case when $\frakl$ is inert in $\Ktilde.$  By Proposition~\ref{prop:Splitting} this coincides with the case when $\ell$ is inert in at least one of $\OO_{d_u}$ and $\OO_{d_x}$.  First assume that $v_{\frakl}(\frakN)$ is even; by Lemma~\ref{lem:Valuations}, $v_{\ell}(N)$ is also even.  Then, by Theorems~\ref{thm:BYLocal} and~\ref{thm:LVLocal}, both sides of~\eqref{eq:BYLVSummandEquality} are $0$.

			Now suppose that $\frakl$ is inert in $\Ktilde$ and that $v_{\frakl}(\frakN)$ is odd.  By Lemma~\ref{lem:Valuations}, if $\frakl'|\ell$ is a prime in $\Ftilde$ different from $\frakl$, then $\frakl'^{-1}\frakN\frakD_{\Ktilde/\Ftilde}$ is \emph{not} integral.  Therefore, $R(\frakl'^{-1}\frakN\frakD_{\Ktilde/\Ftilde}) = 0$ and the right-hand side of~\eqref{eq:BYLVSummandEquality} reduces to
			\[
				\frac{v_{\frakl}(\frakN) + 1}{2}\cdot f(\frakl/\ell)\cdot 
				R_{\Ktilde/\Ftilde}
				(\frakl^{-1}\frakN\frakD_{\Ktilde/\Ftilde}).
			\] 

			Using Theorems~\ref{thm:BYLocal} and~\ref{thm:LVLocal}, Corollary~\ref{corr:mu}, and Lemma~\ref{lem:Valuations} we deduce
			\begin{align*}
				\frac{v_{\frakl}(\frakN) + 1}{2}\cdot f(\frakl/\ell)\cdot  
					R_{\Ktilde/\Ftilde}
					(\frakl^{-1}\frakN\frakD_{\Ktilde/\Ftilde}) & = 
					\frac{v_{\frakl}(\frakN) + 1}{2}
					\prod_{p|N, p\neq \ell}
					\varepsilon_{\Ktilde/\Ftilde}(p,  \frakN \frakD_{\Ktilde/\Ftilde})\\
				\mu(n) R_{d_u}(\ell^{-1}N)\tilde{\rho}_{d_u}(N)
				& = \frac{v_{\ell}(N) + 1}2\prod_{p|N, p\neq \ell}
								\varepsilon_{d_u}(p, N).
			\end{align*}
			We apply Lemma~\ref{lem:Valuations} to show that $\frac{v_{\frakl}(\frakN) + 1}{2} = \frac{v_{\ell}(N) + 1}2$ and Proposition~\ref{prop:LocalFactorEquality} to give $$\varepsilon_{\Ktilde/\Ftilde}(p,  \frakN \frakD_{\Ktilde/\Ftilde}) = \varepsilon_{d_u}(p, N).$$ This completes the proof of Theorem~\ref{thm:SummandTheorem}.\qed

		\subsection{Proof of Theorem~\ref{thm:Main}}
		\label{subsec:ProofOfMainTheorem}
			Theorem~\ref{thm:Main} follows immediately from Proposition~\ref{prop:EquivCondOnn} and Theorem~\ref{thm:SummandTheorem}.\qed

		\subsection{Proof of Corollary~\ref{cor:BYconj}}
		\label{subsec:ProofOfBYconj}
			By Lemma~\ref{lem:pDoesntDivideDelta}, the assumptions of Theorem~\ref{thm:Main} imply the assumptions of Theorem~\ref{thm:LVFormula}.  Thus, Theorems~\ref{thm:LVFormula} and~\ref{thm:Main} complete the proof.\qed


\end{document}